\documentclass{article}
\usepackage[utf8]{inputenc}
\usepackage{appendix}
\usepackage{fullpage}
\usepackage{amsmath}
\usepackage[dvipsnames]{xcolor}
\usepackage{hyperref,todonotes}
\usepackage{graphicx} 
\usepackage{booktabs}
\usepackage{multirow}
\usepackage{tabularx}
\usepackage{mathtools,bm}
\usepackage{amsthm,amssymb}
\usepackage{float}
\usepackage{hyperref}
\usepackage{authblk}

\def\xx{{\bf x}}
\def\yy{{\bf y}}

\def\pp{{\bf p}}

\def\bettaa{{\boldsymbol{\beta}}}
\def\alphaa{{\boldsymbol{\alpha}}}
\def\lambdaa{{\boldsymbol{\lambda}}}

\def\ii{\mathnormal{i}}

\def\jj{\mathnormal{j}}

\newtheorem{thm}{Theorem}

\newtheorem{rem}{Remark}
\newtheorem{defn}{Definition}

\newtheorem{prp}{Proposition}

\newcommand{\mk}[1]{\textcolor{red}{MK: #1}}

\title{Moving Least Squares Approximation using Variably Scaled Discontinuous Weight Function}

\author{Mohammad Karimnejad Esfahani \thanks{\href{}{mohammad.karimnejadesfahani@studenti.unipd.it}}, Stefano De Marchi \thanks{\href{}{stefano.demarchi@unipd.it}}, and Francesco Marchetti \thanks{\href{}{francesco.marchetti@unipd.it}}}

\affil[]{Department of Mathematics "Tullio Levi-Civita", University of Padova}

\date{\today}

\begin{document}

\maketitle

\large
\begin{abstract}
    Functions with discontinuities appear in many applications such as image reconstruction, signal processing, optimal control problems, interface problems, engineering applications and so on. Accurate approximation and interpolation of these functions are therefore of great importance. In this paper, we design a moving least-squares approach for scattered data approximation that incorporates the discontinuities in the weight functions. The idea is to control the influence of the data sites on the approximant, not only with regards to their distance from the evaluation point, but also with respect to the discontinuity of the underlying function. We also provide an error estimate on a suitable {\it piecewise} Sobolev Space. The numerical experiments are in compliance with the convergence rate derived theoretically.
\end{abstract}
\section{Introduction}
In practical applications, over a wide range of studies such as surface reconstruction, numerical solution of differential equations and kernel learning \cite{Bayona,Cuomo,Guastavino}, one has to solve the problem of reconstructing an unknown function $f : \Omega \longrightarrow \mathbb{R}$ sampled at some finite set of data sites $X = \{ \xx_{\ii} \}_{1 \leq \ii \leq N} \subset \Omega \subset \mathbb{R}^d$ with corresponding data values $f_{\ii} = f(\xx_{\ii}), \; 1 \leq \ii \leq N$. Since in practice the function values $f_{\ii}$ are sampled at scattered points, and not at a uniform grid, Meshless (or meshfree) Methods (MMs) are used as an alternative of numerical mesh-based approaches, such as Finite Elements Method (FEM) and Finite Differences (FD). The idea of MMs could be traced back to \cite{Lucy}. Afterwards, multivariate MMs existed under many names and were used in different contexts; interested readers are referred to \cite{Nguyen} for an overview over MMs. In a general setting, MMs are designed, at least partly, to avoid the use of an underlying mesh or triangulation. The approximant of $f$ at $X$ can be expressed in the form
\begin{equation} \label{meshless approx}
    s_{f,X}(\xx) = \sum_{\ii=1}^{N} \alpha_{\ii}(\xx) f_{\ii}.
\end{equation}
One might seek a function $s_{f,X}$ that interpolates the data, i.e. $s_{f,X}(\xx_{\ii}) = f_{\ii}, \; 1 \leq \ii \leq N$, and in this case $\alpha_{\ii}(\xx)$ will be the \textit{cardinal functions}. However, one might consider a more generalized framework known as \textit{quasi-interpolation} in which $s_{f,X}$ only approximates the data, i.e., $s_{f,X}(\xx_{\ii}) \approx f_{\ii}$. The latter case means that we prefer to let the approximant $s_{f,X}$ only nearly fits the function values. This is useful, for instance, when the given data contain some noise, or the number of data is too large. The standard approach to deal with such a problem is to compute the Least-Squares (LS) solution, i.e., one minimizes the error (or cost) function 
\begin{align}\label{LS}
    \sum_{\ii=1}^{N} [s_{f,X}(\xx_{\ii}) - f_{\ii}]^2.
\end{align}
A more generalized setting of LS is known as the weighted LS \mk{requires a reference ?}, in which \eqref{LS} turns to
\begin{equation} \label{WLS}
    \sum_{\ii=1}^{N} [s_{f,X}(\xx_{\ii}) - f_{\ii}]^2 w(\xx_{\ii}),
\end{equation}
which is ruled by the \textit{weighted} discrete $\ell_{2}$ inner product. In practice the role of $w(\xx_{\ii})$ is to add more flexibility to the LS formulation for data $f_{\ii}$ that influence the approximation process, which are supposed, for example, to be affected by some noise. However, these methods are global in the sense that all data sites have influence on the solution at any evaluation point $\xx \in \Omega$. Alternatively, for a fixed evaluation point $\xx$, one can consider only $n$-th closest data sites $\xx_{\ii}, \, \ii=1,\ldots,n$ of $\xx$ such that $n \ll N$. The \textit{Moving Least-Squares} (MLS) method, which is a \textit{local} variation of the classical weighted least-squares technique, has been developed following this idea. To be more precise, in the MLS scheme, for each evaluation point $\xx$ one needs to solve a \textit{weighted least-squares} problem, minimizing
\begin{equation}\label{MLS}
    \sum_{\ii=1}^{N} [s_{f,X}(\xx_{\ii}) - f_{\ii}]^2 w(\xx,\xx_{\ii})
\end{equation}
by choosing the weight functions $w(\xx,\xx_{\ii}) : \mathbb{R}^d \times \mathbb{R}^d \longrightarrow \mathbb{R}$ to be localized around $\xx$, so that \textit{few} data sites are taken into account. The key difference with respect to \eqref{WLS} is that the weight function is indeed \textit{moving} with the evaluation point, meaning that it depends on both the $\xx_{\ii}$ and $\xx$. Consequently, for each evaluation point $\xx$, a small linear system needs to be solved. Also, one can let $w(\cdot,\xx_{\ii})$ be a radial function i.e., $w(\xx, \xx_{\ii}) = \varphi(\lVert \xx - \xx_{\ii} \lVert_{2})$ for some non-negative univariate function $\varphi:[0,\infty)\longrightarrow \mathbb{R}$. Doing in this way, $w(\cdot,\xx_{\ii})$ inherits the translation invariance property of radial basis functions. We mention that \eqref{MLS} could be generalized as well by letting $w_{\ii}(\cdot)=w(\cdot,\xx_{\ii})$ moves with respect to a {\it reference} point $\yy$ such that $ \yy \neq \xx$.

The earliest idea of MLS approximation technique can be traced back to Shepard's seminal paper \cite{ref-proceeding}, in which the author considered the approximation by constants. Later on, the general framework of MLS was introduced by Lancaster and Salkauskas in \cite{ref-journal7}, where they presented the analysis of MLS methods for smoothing and interpolation of scattered data. Afterwards, in \cite{ref-journal5} the author analyzed the connection between MLS and the Backus-Gilbert approach \cite{ref-journal15}, and showed that the method is effective for derivatives approximations as well. Since then, MLS method showed its effectiveness in different applications \cite{ref-journal13, ref-journal14}. The error analysis of MLS approximation has been provided by some authors, mainly based on the work of Levin \cite{ref-journal2}. In \cite[Chap. 3 \& 4]{ref-book1} and \cite{ref-journal3} the author suggested error bounds that take into account the so-called \textit{fill-distance}, whose definition is recalled in Subsection \ref{sec:2.1}. Other works focusing on the theoretical aspects of MLS method include \cite{ref-journal8}, in which the authors provided error estimates in $L_{\infty}$ for the function and its first derivatives in the one dimensional case, then \cite{ref-journal9}, where they generalized this approach to the multi-dimensional case. In both these works, the error analysis is based on the \textit{support} of the weight functions and not on the fill distance. More recently, in \cite{ref-journal10} the author obtained an error estimate for MLS approximation of functions that belong to integer or fractional-order Sobolev spaces, which shows similarities to the bound previously studied in \cite{ref-journal11} for kernel-based interpolation. 

The MLS method has rarely been used for approximating piecewise-continuous functions, i.e, functions that possess some discontinuities or jumps. In this case, it would be essential that the approximant takes into account the location of the discontinuities. To this end, in this paper we let the weight function be a \textit{Variably Scaled Discontinuous Kernel} (VSDK) \cite{ref-journal1}. VSDK interpolant have been employed to mitigate the Gibbs phenomenon, outperforming classical kernel-based interpolation in \cite{ref-journal12}. 
Similarly in MLS approximation framework, the usage of VSDK weights allows the construction of data-dependent approximants (as discussed in \cite[\S4]{ref-journal2}) that are able to overcome the performances of classical MLS approximants, as indicated by a careful theoretical analysis and then assessed by various numerical experiments. 

The paper is organized as follows. In Section \ref{sec2} we recall necessary notions of the MLS, VSDKs and Sobolev spaces. Section \ref{sec3} presents the original contribution of this work, consisting in the use of variably scaled discontinuous weights for reconstructing discontinuous functions in the framework of MLS approximation. The error analysis shows that the MLS-VSDKs approximation can outperform classical MLS schemes as the discontinuities of the underlying function are assimilated into the weight function. In Section \ref{sec4} we discuss some numerical experiments that support our theoretical findings, and in Section \ref{sec5} we draw some conclusions.

\section{Preliminaries on MLS and VSKs} \label{sec2}
\subsection{Moving Least Squares (MLS) approximation}\label{sec:2.1}

In this introduction to MLS, we resume and deepen what outlined in the previous section. The interested readers are also refereed to \cite[Chap. 22]{ref-book2}.

Let $ \Omega $ be a non-empty and bounded domain in $\mathbb{R}^d$ and $X$ be the set of $N$ distinct data sites (or centers). We consider the target function $f$, and the corresponding function values $f_{\ii}$ as defined above. Moreover, $\mathcal{P}^{d}_{\ell}$ indicates the space of $d$-variate polynomials of degree at most $\ell \in \mathbb{N}$, with basis $ \{p_1, ..., p_{Q} \}$ and dimension $Q= {\ell +d \choose d}$. 

Several equivalent formulations exist for the MLS approximation scheme. As the standard formulation, the MLS approximant looks for the best weighted approximation to $f$ at the evaluation point $\xx$ in $\mathcal{P}^{d}_{\ell}$ (or any other linear space of functions $\mathcal{U}$), with respect to the discrete $\ell_2$ norm induced by the weighted inner product $ \langle f,g \rangle_{w_{\xx}} = \sum_{\ii=1}^{N} w(\xx_{\ii}, \xx) f(\xx_{\ii}) g(\xx_{\ii})$. Mathematically speaking, the MLS approximant will be the linear combination of the polynomial basis i.e.,
\begin{equation}\label{MLS;standard}
    s_{f,X}(\xx) = \sum_{\jj=1}^{Q} c_{\jj}(\xx) p_{\jj}(\xx),
\end{equation}
where the coefficients are obtained by locally minimizing the weighted least square error in \eqref{MLS}, which is equivalent to minimizing $\lVert f - s_{f} \lVert_{w_{\xx}}$. We highlight that the local nature of the approximant is evident from the fact that the coefficient $c_{\jj}(\xx)$ must be computed for each evaluation point $\xx$.

In another formulation of MLS approximation known as the \textit{Backus-Gilbert} approach, one considers the approximant $s_{f,X}(\xx)$ to be a \textit{quasi interpolant} of the form \eqref{meshless approx}. In this case, one seeks the values of the basis functions $\alpha_{\ii}(\xx)$ (also known as generating or shape functions) as the minimizers of 
\begin{equation}
    \frac{1}{2} \sum_{\ii=1}^{N} \alpha_{\ii}^2(\xx) \frac{1}{w(\xx_{\ii},\xx)}
\end{equation}
subject to the polynomial reproduction constraints \begin{equation*}
    \sum_{\ii=1}^{N} p(\xx_{\ii}) \alpha_{\ii}(\xx) = p(\xx), \quad \text{for all} \; p \in \mathcal{P}_{\ell}^d.
\end{equation*}
Such a constrained quadratic minimization problem can be converted to a
system of linear equations by introducing Lagrange multipliers $ \boldsymbol{\lambda}(\xx) = [\lambda_{1}(\xx),...,\lambda_{Q}(\xx)]^T$. Consequently (e.g see \cite[Corollary 4.4]{ref-book1}), the MLS basis function $\alpha_{\ii} $ evaluated at $\xx$ is given by 
\begin{equation} \label{shape;fun}
    \alpha_{\ii}(\xx) = w(\xx, \xx_{\ii}) \sum_{k=1}^{Q}\lambda_{k}(\xx)p_{k}(\xx_{\ii}), \quad 1 \leq \ii \leq N,
\end{equation}
such that $\lambda_{k}(\xx)$ are the unique solution of 
\begin{equation} \label{lagrange;mult}
    \sum_{k=1}^{Q}  \lambda_{k}(\xx) \sum_{\ii=1}^{N}w(\xx, \xx_{\ii}) p_{k}(\xx_{\ii})p_{s}(\xx_{\ii}) = p_{s}(\xx), \quad  1 \leq s \leq Q.
\end{equation}
We observe that the weight function $w_{\ii}(\xx)=w(\xx,\xx_{\ii})$ controls the influence of the center $\xx_{\ii}$ over the approximant, so it should be \textit{small} when evaluated at a point that is far from $\xx$, that is it should decay to zero fast enough. To this end we may let $w_{\ii}(\xx)$ be positive on a ball centered at $\xx$ with radius $r$, $B(\xx,r)$, and zero outside. For example, a compactly supported radial kernel satisfies such a behaviour. Thus, let $I(\xx) = \{ \ii \in \{1,\ldots,N\}, \lVert \xx - \xx_{\ii} \lVert_2 \leq r \}$ be the family of indices of the centers $X$, for which $w_{\ii}(\xx) > 0$, with $\lvert I \lvert = n \ll N$. Only the centers $\xx_{\ii} \in I$ influence the approximant $s_{f,X}(\xx)$. Consequently, the matrix representation of \eqref{shape;fun} and \eqref{lagrange;mult} is
\begin{align*}
    \alphaa (\xx) &= W(\xx) P^T \lambdaa(\xx), \\
    \lambdaa(\xx) &= (PW(\xx)P^T)^{-1} \pp(\xx),
\end{align*} 
where $\alphaa (\xx) = [\alpha_{1}(\xx),...,\alpha_{n}(\xx)]^T$, $W(\xx) \in \mathbb{R}^{n \times n}$ is the diagonal matrix carrying the weights $w_{\ii}(\xx)$ on its diagonal, $P \in \mathbb{R}^{Q \times n}$ such that its $k$-th row contains $p_{k}$ evaluated at data sites in $I(\xx)$, and $\pp(\xx) = [p_{1}(\xx), ..., p_{Q}(\xx)]^T$. More explicitly the basis functions are given by
\begin{equation} \label{shap;fun;mat;form}
    \alphaa (\xx) = W(\xx) P^T (PW(\xx)P^T)^{-1} \pp(\xx).    
\end{equation}
Moreover, it turns out that the solution of \eqref{MLS;standard} is identical to the solution offered by the Backus-Gilbert approach (see e.g. \cite[Chap. 3 \& 4]{ref-book1}).

In the MLS literature, it is known that a local polynomial basis shifted to the evaluation point $\xx \in \Omega$ leads to a more stable method (see e.g. \cite[Chap. 4]{ref-book1}). Accordingly, we let the polynomial basis to be $ \{ 1,(\cdot - \xx ),\dots,(\cdot - \xx )^{\ell} \}$, meaning that different bases for each evaluation point are employed. In this case, since with standard monomials basis we have $p_{1} \equiv 1$ and $p_{k}(0) = 0$ for $ 2 \leq k \leq Q$, then $\pp(\xx) = [1, 0, ..., 0]^T$. 

To ensure the invertibility of $PW(\xx)P^T$ in \eqref{shap;fun;mat;form}, $X$ needs to be $\mathbb{P}^d_{\ell}$-unisolvent. 
Then as long as $w_{\ii}(\xx)$ is positive, $PW(\xx)P^T$ will be a positive definite matrix, and so invertible; more details are available in \cite[Chap. 22]{ref-book2}.
 
Furthermore, thanks to equation \eqref{shape;fun}, it is observable that the behaviour of $\alpha_{\ii}(\xx)$ is heavily influenced by the behaviour of the weight functions $w_{\ii}(\xx)$, in particular it includes continuity and the support of the basis functions $\alpha_{\ii}(\xx)$. Another significant feature is that the weight functions $w_{\ii}(\xx)$ which are singular at the data sites lead to cardinal basis functions i.e., $\alpha_{\ii}(\xx_{\jj}) = \delta_{\ii, \jj} \; \ii, \jj = 1, ..., n$, meaning that MLS scheme interpolates the data (for more details see \cite[Theorem 3]{ref-journal2}).

We also recall the following definitions that we will use for the error analysis.
\begin{enumerate} \label{def1} 
    \item  A set $\Omega \subset \mathbb{R}^d$ is said to satisfy an \textbf{interior cone condition} if there exists an angle $\Theta \in (0, \pi/2)$ and a radius $r > 0 $ so that for every $\xx \in \Omega$ a unit vector $\xi(\xx)$ exists such that the cone $$ C(\xx, \xi, \Theta, r) = \{ \xx + t\yy : \yy \in \mathbb{R}^d, \lVert \yy \lVert_{2} = 1, \cos(\Theta) \leq \yy^T \xi, t \in [0,r] \} $$ is contained in $\Omega$.
    \item A set $X = \{\xx_{1}, ...,\xx_{N} \}$ with $Q \leq N$ is called $\mathbb{P}^d_{\ell}$-unisolvent if the zero polynomial is the only polynomial from $\mathbb{P}^d_{\ell}$ that vanishes on $X$.
    \item The \textbf{fill distance} is defined as $$ h_{X, \Omega} = \underset{\xx \in \Omega}{\sup} \underset{1 \leq j \leq N}{\min} \lVert \xx - \xx_{j} \lVert_{2}.$$
    \item The \textbf{separation distance} $$ q_{X} = \frac{1}{2} \underset{i \ne j}{\min} \lVert \xx_{i} - \xx_{j} \lVert.$$
    \item The set of data sites $X$ is said to be \textbf{quasi-uniform} with respect to a constant $c_{qu} > 0$ if $$ q_{X} \leq h_{X, \Omega} \leq c_{qu}q_{X}.$$
    \end{enumerate}


\subsection{Sobolev spaces and error estimates for MLS}
Assume $k \in \mathbb{N}_{0}$ and $p \in [1, \infty) $, then the \textit{integer-order} Sobolev space $W_{p}^{k}(\Omega)$ consists of all $u$ with distributional (weak) derivatives $D^{\boldsymbol{\delta}} u \in L^{p}, \lvert \boldsymbol{\delta} \lvert \leq k$. The semi-norm and the norm associated with these spaces are
\begin{equation} \label{sob;norms}
    \lvert u \lvert_{w_{p}^{k}(\Omega)} := \Big( \sum_{\lvert \boldsymbol{\delta} \lvert = k} \lVert D^{\boldsymbol{\delta}} u \lVert_{L^p(\Omega)}^{p} \Big)^{1/p} \;\;, \qquad \lVert u \lVert_{w_{p}^{k}(\Omega)} := \Big( \sum_{\lvert \boldsymbol{\delta} \lvert \leq k} \lVert D^{\boldsymbol{\delta}} u \lVert_{L^p(\Omega)}^{p} \Big)^{1/p}.
\end{equation}
Moreover, letting $0 < s < 1$, the \textit{fractional-order} Sobolev space $W_{p}^{k+s}(\Omega)$ is the space of the functions $u$ for which semi-norm and norm are defined as 
\begin{align*}
    \lvert u \lvert_{w_{p}^{k+s}(\Omega)} &:= \Big( \sum_{ \lvert \boldsymbol{\delta} \lvert = k} \int_{\Omega} \int_{\Omega} \frac{ \lvert D^{\boldsymbol{\delta}} u(\xx) - D^{\boldsymbol{\delta}} u(\yy) \lvert^p }{\lvert \xx - \yy \lvert^{d + ps}} \Big)^{1/p}\\
    \lVert u \lVert_{W_{p}^{k+s}(\Omega)} &:= \Big( \lVert u \lVert_{W_{p}^{k}(\Omega)} + \lvert u \lvert_{W_{p}^{k+s}(\Omega)} \Big)^{1/p}.
\end{align*}
Consider certain Sobolev spaces $ W_{p}^{k}(\Omega) $ with the condition that $1 < p < \infty $ and $k > m + d/p$ (for $p = 1$ the equality is also possible), then according to \cite[Theorem 2.12]{ref-journal11} the sampling inequality  $$ \lVert u \lVert_{W_{p}^m(\Omega)} \leq C h_{X,\Omega}^{k - m -d(1/p - 1/p)_{+}} \lVert u \lVert_{W_{p}^k} $$ holds for a function $u$ that satisfies $u(X) = 0$, with $h_{X, \Omega}$ being the \textit{fill distance} associated with $X$ and $(\yy)_{+} = \max{ \{0,\yy \} }$. For more information regarding Sobolev Spaces and sampling inequalities we refer the reader to \cite{ref-book3} and \cite{ref-thesis1}, respectively.

Getting back to the MLS scheme, let $D^{\boldsymbol{\delta}}$ be a derivative operator such that $\lvert \boldsymbol{\delta} \lvert \leq \ell $ (we recall that $\ell$ is the maximum degree of the polynomials).  Under some mild conditions regarding the weight functions, \cite[Theorem 3.11]{ref-journal10} shows that $\{D^{\boldsymbol{\delta}} \alpha_{\ii}(\xx) \}_{1 \leq \ii \leq n}$ forms a \textit{local polynomial reproduction} in a sense that there exist constants $h_{0}, \; C_{1,\boldsymbol{\delta}}, \; C_{2}$ such that for every evaluation point $\xx$ \begin{itemize}
    \item $\sum_{\ii=1}^{N} D^{\boldsymbol{\delta}} \alpha_{\ii}(\xx) p(\xx_{\ii}) = p(\xx)$ for all $ p \in \mathbb{P}_{\ell}^{d}$
    \item $\sum_{\ii=1}^{N} \lvert D^{\boldsymbol{\delta}} \alpha_{\ii}(\xx) \lvert \leq C_{1,\boldsymbol{\delta}} h_{X,\Omega}^{-\lvert \boldsymbol{\delta} \lvert}$
    \item $ D^{\boldsymbol{\delta}} \alpha_{\ii}(\xx) = 0$ provided that $ \lVert \xx - \xx_{\ii} \lVert_{2} \geqslant C_{2} h_{X,\Omega} $
\end{itemize}
for all $X$ with $h_{X,\Omega} \leq h_{0}$.

The particular case of $\lvert \boldsymbol{\delta} \lvert = 0$ was previously discussed in \cite[Theorem 4.7]{ref-book1} in which it is shown that $\{ \alpha_{\ii}(\xx) \}_{1 \leq \ii \leq n}$ forms a local polynomial reproduction. However in this case the basis functions $\{ \alpha_{\ii}(\cdot) \}_{1 \leq \ii \leq n}$ could be even discontinuous but it is necessary that $w_{\ii}(\xx)$ are bounded (for more details see \cite[Chap 3,4]{ref-book1}). Consequently we restate the the MLS error bound in Sobolev Spaces developed in \cite{ref-journal10}.
\begin{thm} \cite[Theorem 3.12]{ref-journal10}
Suppose that $\Omega \subset \mathbb{R}^d$ is a bounded set with a Lipschitz boundary. Let $\ell$ be a positive integer, $0 \leq s < 1, \; p \in [1, \infty) \; ,q \in [1,\infty]$ and let $\boldsymbol{\delta}$ be a multi-index satisfying $ \ell > \lvert \boldsymbol{\delta} \lvert + d/p$ for $p > 1$ and $ \ell \geqslant \lvert \boldsymbol{\delta} \lvert + d $ for $p =1$. If $f \in W^{\ell+s}_{p}(\Omega)$, there exist constants $C > 0$  and $h_{0}> 0$ such that for all $X = \{\xx_{1}, ..., \xx_{N} \} \subset \Omega$ which are quasi-uniform with $h_{X,\Omega} \leq \min \{h_{0},1 \} $, the error estimate holds
\begin{equation} \label{MLS;err;bound}
    \lVert f - s_{f,X} \lVert_{W_q^{\lvert \boldsymbol{\delta} \lvert}(\Omega)} \leq C h _{X,\Omega}^{\ell+s - \lvert \boldsymbol{\delta} \lvert - d (1/p - 1/q)_{+}}  \lVert f \lVert_{W_{p}^{\ell+s}(\Omega)}. 
\end{equation}
when the polynomial basis, are shifted to the evaluation point $\xx$ and scaled with respect to the fill distance $h_{X,\Omega}$, and $w_{\ii}(\cdot)$ is positive on $[0, 1/2]$, supported in $[0,1]$ such that its even extension is non negative and continuous on $\mathbb{R}$.
\label{thm;Mirzaei}
\end{thm}
\begin{rem}
The above error bounds holds also when $s=1$. However, recalling the definition of (semi-)norms in \textit{fractional-order} Sobolev space, we see that in this case we reach to an \textit{integer-order} Sobolev space of $\ell+1$. Therefore, it requires that $\ell+1 > \lvert \boldsymbol{\delta} \lvert + d/p$ for $p>1$ or $\ell+1 \geqslant \lvert \boldsymbol{\delta} \lvert$ for $p=1$ in order that \eqref{MLS;err;bound} holds true. The key point is that in this case, the polynomial space is still $\mathcal{P}_{\ell}^d$ and not $\mathcal{P}_{\ell+1}^d$.
\end{rem}
\subsection{Variably Scaled Discontinuous Kernels (VSDKs)} 
Variably Scaled Kernels (VSKs) were firstly introduced in \cite{ref-journal4}. The basic idea behind them is to map the data sites from $\mathbb{R}^d$ to $\mathbb{R}^{d+1}$ via a scaling function $\psi : \Omega \longrightarrow \mathbb{R}$ and to construct an augmented approximation space in which the data sites are $\{ (\xx_{\ii}, \psi(\xx_{\ii})) \; \ii = 1, ..., N \}$ (see \cite[Def. 2.1]{ref-journal4}). Though the first goal of doing so was getting a \textit{better} nodes distribution in the augmented dimension, later on in \cite{ref-journal1} the authors came up with the idea of also encoding the behaviour of the underlying function $f$ inside the scale function $\psi$. Precisely, for the target function $f$ that possesses some jumps, the key idea is the following. 
\begin{defn} \label{def;scale;fun}
Let $\mathcal{P} = \{\Omega_{1}, ..., \Omega_{n} \}$ be a partition of $\Omega$ and let $\bettaa=(\beta_{1}, ..., \beta_{n})$ be a vector of real distinct values. Moreover, assume that all the jump discontinuities of the underlying function $f$ lie on $\bigcup_{\jj=1}^n{\partial \Omega_{\jj}}$. The piecewise constant scaling function $\psi_{\mathcal{P},\bettaa}$ with respect to the partition $\mathcal{P}$ and the vector $\bettaa$ is defined as $$\psi_{\mathcal{P},\bettaa}(\xx)|_{\Omega_{\jj}} = \beta_{\jj},\;\xx\in\Omega.$$ Successively, let $\Phi^{\varepsilon}$ be a positive definite radial kernel on $\Omega \times \Omega$ that depends on the shape parameter $\varepsilon > 0$. A variably scaled discontinuous kernel on $(\Omega \times \mathbb{R}) \times  (\Omega \times \mathbb{R})$ is defined as 
\begin{equation} \label{def;VSDK}
    \Phi^{\varepsilon}_{\psi}(\xx, \yy) = \Phi^{\varepsilon} \big( \Psi(\xx), \Psi(\yy) \big),\quad \xx, \yy \in \Omega.
\end{equation}
such that $\Psi(\xx) = (\xx, \psi(\xx))$.
\end{defn}
Moreover, we point out that if $\Phi^{\varepsilon}$ is (strictly) positive definite then so is $\Phi^{\varepsilon}_{\psi}$, and if $\Phi^{\varepsilon}$ and $\psi$ are continuous then so is $\Phi^{\varepsilon}_{\psi}$ \cite[Theorem 2.2]{ref-journal4}. Figure~\ref{fig1} shows two different choices for the discontinuous scale function for the univariate case. In any case, it matters that the discontinuities of the target function $f$ are assimilated into the kernel $\Phi^{\varepsilon}_{\Psi}$.
\begin{figure}[h]
\includegraphics[width = 0.5\textwidth]{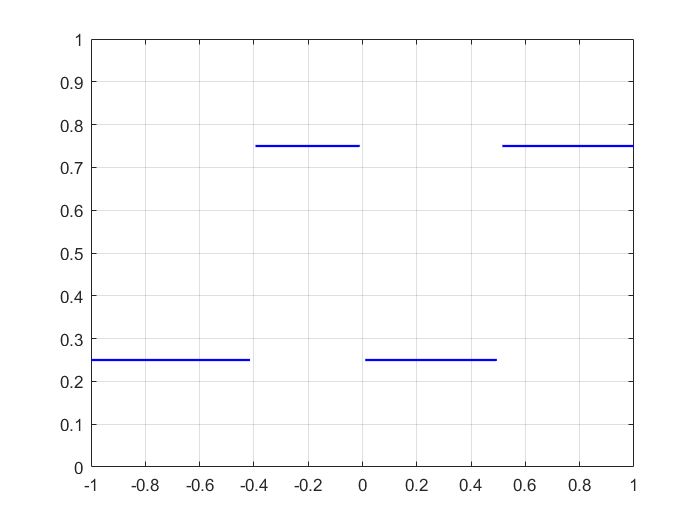}
\hfill
\includegraphics[width = 0.5\textwidth]{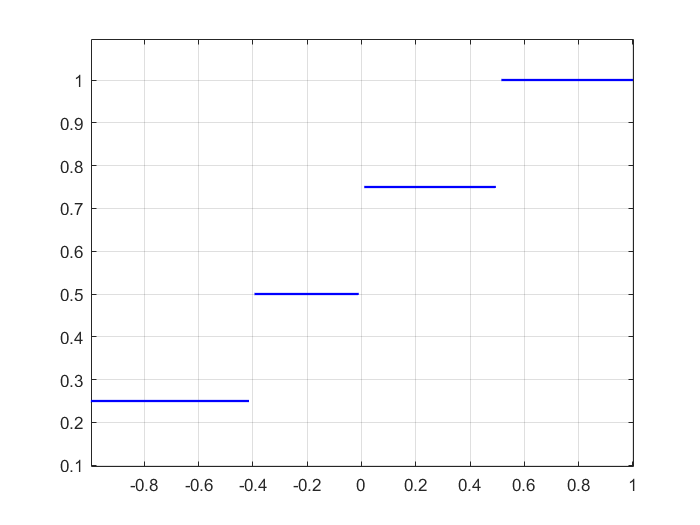}
\caption{ Discontinuous scale functions.}
\label{fig1}
\end{figure}
\section{MLS-VSDKs} \label{sec3}
Let $f$ be a function with some jump discontinuities defined on $\Omega$, $\mathcal{P}$ and $\psi_{\mathcal{P},\beta}$ as in Definition \ref{def;scale;fun}. We look for the MLS approximant with \textit{variably scaled discontinuous weight function} such that
\begin{equation} \label{VSDK;weight}
w_{\psi}(\xx,\xx_{\ii})= w(\Psi(\xx), \Psi(\xx_{\ii})).
\end{equation}
Above all, we point out that in this case the diagonal matrix $W(\xx)$ in \eqref{shap;fun;mat;form} still carries only positive values by assumption, and therefore the equation \eqref{shap;fun;mat;form} is still solvable meaning that the basis functions $\alpha(\xx)$ uniquely exist. However, with new weight functions, from \eqref{VSDK;weight} also $\alpha(\xx)$ might be continuous or discontinuous regarding to the given data values $f_{\ii}$. Therefore our basis functions are indeed data-dependent thanks to \eqref{VSDK;weight}. From now on, we call this scheme MLS-VSDK, and we will denote the corresponding approximant as $s_{f,X}^{\psi}$.  

Since the basis functions are data dependent, one might expect that the space in which we express the error bound should be data dependent as well. Towards this idea, for $k \in \mathbb{Z}, \; 0 \leq k$, and $1 \leq p \leq \infty$, we define the \textit{piecewise} Sobolev Spaces
\begin{equation*}
    {\cal W}^k_{p} (\Omega) = \{ f: \Omega \longrightarrow \mathbb{R} \; \text{s.t.} \; f_{\lvert_{\Omega_{\jj}}} \in W_{p}^{k}(\Omega_{\jj}), \quad  \jj \in \{ 1,..., n \} \},
\end{equation*}
where $f_{\lvert_{\Omega_{\jj}}}$ denotes the restriction of $f$ to $\Omega_{\jj}$, and $W_{p}^k(\Omega_{\jj})$ denote the Sobolev space on $\Omega_{\ii}$. We endow ${\cal W}^k_{p} (\Omega)$ with the norm
\begin{equation} \label{piece;norm}
    \lVert f \lVert_{{\cal W}_{p}^{k} (\Omega)} = \sum_{\jj=1}^{n} \lVert f 
    \lVert_{W_{p}^k(\Omega_{\jj})}.
\end{equation}
When $k = 0$ we simply denote ${\cal W}_{p}^{0}(\Omega) $ by $ {{\cal L}^p}(\Omega)$. Moreover, it could be shown that for any partition of $\Omega$ the standard Sobolev space $W_{p}^k(\Omega)$ is contained in ${\cal W}_{p}^{k}(\Omega)$ (see \cite{ref-journal12} and reference therein). We assume that every set $\Omega_{\jj} \in \mathcal{P}$ satisfies Lipschitz boundary conditions which will be essential for our error analysis. 
\begin{prp} \label{prop}
    Let $\mathcal{P}$ be as in Definition \ref{def;scale;fun} and set the derivative order $\boldsymbol{\delta} = 0$. Then, by using Theorem \eqref{thm;Mirzaei}, the error satisfies the inequality
    \begin{equation}   \label{sub;dom;bound}
    \lVert f - s_{f,X}^{\psi} \lVert_{L^2(\Omega_{\jj})} \leq C_{\jj} h _{\Omega_{\jj}}^{\ell+1 - d (1/p - 1/2)_{+}}  \lVert f \lVert_{W_{p}^{\ell+1}(\Omega_{\jj})}, \quad\quad \text{for all} \;\; \Omega_{\jj} \in \mathcal{P}
    \end{equation}
    with $h_{\Omega_{\jj}} $ the fill distance with respect to $\Omega_{\jj}$.
\end{prp}
\begin{proof}
Recalling Definition \ref{def;scale;fun} we know that the discontinuities of $f$ and subsequently $w_{\ii}(\cdot)$ are located only at the boundary and not on the domain $\Omega_{\jj}$, meaning that $w_{\ii}(\cdot)$ is continuous inside $\Omega_{\jj}$. Furthermore, the basis $\{ \alpha_\ii (\xx) \}_{1 \leq \ii \leq n}$ forms a local polynomial reproduction i.e., there exists a constant $C$ such that $\sum_{\ii=1}^{N} \lvert \alpha_{\ii} \lvert \leq C$. Letting $s = 1$ and $q=2$, by noticing that $W_{q}^{0}(\Omega_{\jj}) = L^q(\Omega_{\jj})$, then the error bound \eqref{sub;dom;bound} is an immediate consequence of Theorem \eqref{thm;Mirzaei}.
\end{proof}
From the above proposition, it could be understood that $s_{f,X}^{\psi}$ behaves similarly to $s_{f,X}$ in the domain $\Omega_{\jj}$, where there is no discontinuity. This is in agreement with Definition \ref{def;scale;fun}. Consequently, it is required to extend the error bound \eqref{sub;dom;bound} to the whole domain $\Omega$.
\begin{thm} \label{VSDK;err;est;thm}
Let $f$, $\mathcal{P}$, $\psi_{\mathcal{P},\beta}$ be as before, and the weight functions as in \eqref{VSDK;weight}. Then, for $ \ell > \lvert \boldsymbol{\delta} \lvert + d/p$ (equality also holds for $p=1$), and $ f \in {\cal W}_{p}^{\ell+1}(\Omega)$, for the MLS-VSDK approximant $s_{f,X}^{\psi}$ the error can be bounded as follows:
\begin{equation} \label{MLS;VSDK;err}
    \lVert f - s^{\psi}_{f,X} \lVert_{{\cal L}^2(\Omega)} \leq C {h}^{\ell+1 - d (1/p - 1/2)_{+}} \lVert f \lVert_{{\cal W}_{p}^{\ell+1} (\Omega) }
\end{equation}
\end{thm}
\begin{proof}
By Proposition \eqref{prop}, we know that \eqref{sub;dom;bound} holds for each $\Omega_{\jj}$. Let $h_{X, \Omega_{\ii}}$ and $C_{\ii}$ be the fill distance and a constant associated with each $ \Omega_{\ii}$, respectively. Then, we have
\begin{equation*}
    \sum_{\jj=1}^n \lVert f - s_{f,X}^{\psi} \lVert_{L^2(\Omega_{\jj})} \leq \sum_{\jj=1}^n C_{\jj} h_{X, \Omega_{\ii}}^{\ell+1 - d (1/p - 1/2)_{+}}  \lVert f \lVert_{W_{p}^{\ell+1}(\Omega_{\jj})}.
\end{equation*}
By definition we get $ \sum_{\jj=1}^n \lVert f - s_{f,X}^{\psi} \lVert_{L^2(\Omega_{\jj})} = \lVert f - s_{f,X}^{\psi} \lVert_{{\cal L}^2(\Omega)}$. Moreover, letting $ C = \max\{C_{1}, ..., C_{n} \}$ and $h=max\{h_{X,\Omega_{1}}, ..., h_{X,\Omega_{n}} \} $ then the right hand side can be bounded by
\begin{equation*}
    C h^{\ell+1 - d (1/p - 1/2)_{+}} \lVert f \lVert_{{\cal W}_{p}^{\ell+1}(\Omega)}.
\end{equation*}
Putting these together we conclude.
\end{proof}
Some remarks are in order.
\begin{enumerate}
    \item One might notice that the error bound in \eqref{MLS;err;bound} is indeed local (the basis functions are local by assumption), meaning that if $f$ is less smooth in a subregion of $\Omega$, say it possesses only $\ell^{'} \leq \ell$ continuous derivatives there, then the approximant (interpolant) has order $\ell^{'} + 1$ in that region and this is the best we can get. On the other hand according to \eqref{MLS;VSDK;err}, thanks to the definition of piecewise Sobolev space, the regularity of the underlying function in the interior of the subdomain $\Omega_{\jj}$ matters. In other words, as long as $f$ possesses regularity of order $\ell$ in subregions, say $\Omega_{\jj}$ and $\Omega_{\jj+1}$, the approximant order of $\ell + 1$ is achievable, regardless of the discontinuities on the boundary of $\Omega_{\jj}$ and $\Omega_{\jj+1}$.
    \item Another interesting property of the MLS-VSDK scheme is that it is indeed data dependent. To clarify, for the evaluation point $\xx \in \Omega_{\jj}$ take two data sites $\xx_{\ii}, \; \xx_{\ii+1} \in B(\xx,r)$ with the same distance from $\xx$ such that $\xx_{\ii} \in \Omega_{\jj}$ and $\xx_{\ii+1} \in \Omega_{\jj+1}$. Due to the definition (\ref{def;VSDK}), $w_{\psi}(\xx, \xx_{\ii+1})$ decays to zero faster than $w_{\psi}(\xx, \xx_{\ii})$ i.e., the data sites from the same subregion $\Omega_{\jj}$ pay more contribution to the approximant (interpolant) $s_{f,X}^{\psi}$, rather than the one from another subregion $\Omega_{\jj+1}$ beyond a discontinuity line \mk{is line correct for high dimension?}. On the other hand in the classical MLS scheme, this does not happen as the weight function gives the same value to both $\xx_{\ii}$ and $\xx_{\ii+1}$.
    \item We highlight that in MLS-VSDK scheme we do not scale polynomials and so the polynomial space $\mathcal{P}^d_{\ell}$ is not changed. We scale only the weight functions and thus, in case the given function values bear discontinuities, the basis functions $\{ \alpha_{\ii}(\cdot) \}_{1 \leq \ii \leq n}$ are modified. 
\end{enumerate}
We end this section by recalling that the MLS approximation convergence order is achievable only in the \textit{stationary setting}, i.e., the shape parameter $\varepsilon$ must be scaled with respect to the fill distance. It leads to \textit{peaked basis functions} for densely spaced data and \textit{flat basis function} for coarsely spaced data. In other words, the local support of the weight functions $B(\xx,r)$, and subsequently basis functions must be tuned with regards to the $h_{X,\Omega}$ using the shape parameter $\varepsilon$. Consequently, this holds also in MLS-VSDK scheme, meaning that after scaling $w_{\ii}$ we still need to take care of $\varepsilon$. This is different with respect to VS(D)Ks interpolation where $\varepsilon =1$ was kept fixed \cite{ref-journal4, ref-journal1}.


\section{Numerical experiments} \label{sec4}
In this section, we compare the performance of the MLS-VSDK with respect to the classical MLS method. In all numerical tests we fix the polynomials space up to degree $1$. Considering the evaluation points as $Z = \{z_{1}, ..., z_{s}\}$ we compute root mean square error and maximum error by
\begin{equation*}
RMSE = \sqrt{{\frac{1}{s}}\sum_{i=1}^s (f(z_{\ii}) - s_{f,X}(z_{\ii}))^2}, \quad MAE = \underset{z_{\ii} \in Z}{\max} \lvert f(z_{\ii}) - s_{f,X}(z_{\ii}) \lvert.
\end{equation*}
We consider four different weight functions to verify the convergence order of $s_{f,x}^{\psi}$ to a given $f$, as presented in Theorem \ref{VSDK;err;est;thm}.
\begin{enumerate}
    \item 
    $w^1(\xx,\xx_{\ii}) =  (1-\varepsilon \lVert \xx - \xx_{\ii} \lVert)^4_+ \cdot (4\varepsilon \lVert \xx - \xx_{\ii} \lVert + 1)$, which is the well-known $C^2$ \textit{Wendland} function. Since each $w^1_{\ii}$ is locally supported on the open ball $B(0,1)$ then it verifies the conditions required by Theorem \eqref{VSDK;err;est;thm}.
    \item
    $w^2(\xx, \xx_{\ii}) = \exp(- \varepsilon \lVert \xx - \xx_{\ii} \lVert^2)$, i.e. the Gaussian RBF. We underline that when Gaussian weight functions are employed, with decreasing separation distance of the approximation centers, the calculation of the basis functions in \eqref{shap;fun;mat;form} can be badly conditioned. Therefore, in order to make the computations stable,  in this case we regularize the system by adding a small multiple, say $\lambda = 10^{-8}$, of the identity to the diagonal matrix $W$.
    \item
    $w^3(\xx, \xx_{\ii}) = \exp(-\varepsilon \lVert \xx - \xx_{\ii} \lVert)(15 + 15 \lVert \xx - \xx_{\ii} \lVert + 6 \lVert \xx - \xx_{\ii} \lVert^2 + \lVert \xx - \xx_{\ii} \lVert^3)$, that is a \textit{$C^6$ Matérn} function.
    \item
    $w^4(\xx,\xx_{\ii}) = ( \exp{( \varepsilon \lVert \xx - \xx_{\ii} \lVert )}^2 - 1 )^{-1}$, suggested in \cite{ref-journal2}, which enjoys an additional feature which leads to interpolatory MLS, since it possesses singularities at the centers.
\end{enumerate}
One might notice that $w^2$, $w^3$ and $w^4$ are not locally supported. However, the key point is that they are all decreasing with the distance from the centers and so, in practice, one can overlook the data sites that are so far from the center $\xx$. As a result, one generally considers a \textit{local stencil} containing $n$ nearest data sites of the set $Z$ of evaluation points. While there is no clear theoretical background concerning the stencil size, in MLS literature, one generally lets $n = 2 \times Q$ (see e.g \cite{ref-journal6}). However, it might be possible that in some special cases one could reach a better accuracy using different stencil sizes. This aspect is covered by our numerical tests, which are outlined in the following.
\begin{enumerate}
    \item 
    In Section \ref{sub;sec;1}, we present an example in the one-dimensional framework, where the stencil size is fixed to be $n=2 \times Q$. Moreover, we consider $w^1$, $w^2$ and $w^3$.
    \item
    In Section \ref{sub;sec;2}, we move to the two-dimensional framework and we keep the same stencil size. Here, we restrict the test to the weight function $w^1$ and verify Theorem \eqref{VSDK;err;est;thm}.
    \item
    In Section \ref{sub;sec;3}, we remain in the two-dimensional setting but the best accuracy is achieved with $n=20$. Moreover, in addition to $w^2$ and $w^3$, we test the interpolatory case by considering $w^4$ as weight function. \item
    In Section \ref{sub;sec;4}, we present a two-dimensional experiments where the data sites have been perturbed via some white noise. We fix $n=25$ and $w^2,w^3$ are involved.  
\end{enumerate}

\subsection{Example 1} \label{sub;sec;1}

On $\Omega=(-1,1)$, we assess MLS approximant for 
\begin{equation*}
    f_{1}(x) = \begin{cases}
    e^{-x}, \quad -1 < x < -0.5\\
    x^3, \quad -0.5 \leq x < 0.5, \\
    1, \quad 0.5 \leq x < 1
    \end{cases}
\end{equation*}
with discontinuous scale function
    \begin{equation*}
        \psi(x) = \begin{cases} 
        1, \; \; x \in (-1, \-0.5) \; \text{and} \; [0.5, 1) \\
        2, \; \; x \in [-0.5, 0.5)\,.
        \end{cases}
    \end{equation*}
We note that the function $\psi$ is defined only by two cases. The important fact is that has a jump as $f_1.$

To evaluate the approximant consider the evaluation grid of equispaced points with step size $5.0e-4$. Tables \ref{f1;table} and \ref{f1;table2} include RMSE of $f_{1}$ approximation using $w^1$ as the weight function.
\begin{table}[ht] 
\begin{tabular}{ || c | c |  c |  c || }
\hline
\textbf{number of centers}	& \textbf{$\varepsilon$ value} & \textbf{RMSE MLS-VSDK}  &\textbf{RMSE classic MLS}\\
\hline \hline
9		& 0.25		& 3.58e-1	      & 3.95e-1  \\
17		& 0.5		& 1.99e-1 	      & 3.02e-1  \\
33      & 1         & 3.10e-3        & 2.17e-1  \\
65      & 2         & 8.42e-4        & 1.54e-1 \\
257     & 4         & 5.67e-5        & 7.68e-2  \\
513     & 8         & 1.43e-5        & 5.35e-2  \\
\hline
\end{tabular}
\centering
\caption{Comparison of the RMSE for $f_{1}$ approximation at \textit{uniform} data sites.\label{f1;table}}
\centering
\end{table}
\begin{table}[ht] 
\begin{tabular}{ || c | c |  c |  c || }
\hline
\textbf{number of centers}	& \textbf{$\varepsilon$ value} & \textbf{RMSE MLS-VSDK}  &\textbf{RMSE classic MLS}\\
\hline \hline
9		& 0.25		& 3.53e-1	      & 3.77e-1  \\
17		& 0.5		& 1.99e-1 	      & 3.01e-1  \\
33      & 1         & 3.08e-3        & 2.17e-1  \\
65      & 2         & 8.39e-4        & 1.54e-1  \\
257     & 4         & 5.67e-5        & 7.73e-2  \\
513     & 8         & 1.43e-5        & 5.41e-2  \\
\hline
\end{tabular}
\centering
\caption{Comparison of the RMSE for $f_{1}$ approximation at \textit{Halton} data sites.\label{f1;table2}}
\centering
\end{table}

Again, in order to investigate the convergence rate, consider two sets of uniform and Halton nodes with the size from Table \ref{f1;table}. In order to generalize our results to  globally supported weight functions, we take into account $w^{2}$ and $w^{3}$, Gaussian and Mat\'ern $C^6$ radial functions respectively. For the uniform data sites let the shape parameter values to be $\boldsymbol{\varepsilon}_{GA}^{U} = [5, 20, 40, 80, 160, 320] $ and $\boldsymbol{\varepsilon}_{Mat}^{U} = [5, 10, 20, 40, 80, 160]$ for $w^2$ and $w^3$. Our computation shows convergence rates of $ 2.54 $ and $ 2.26 $ for MLS-VSDK scheme, shown in Figure \ref{f1;con;rate}. Accordingly, for \textit{Halton} points let $\boldsymbol{\varepsilon}_{Mat}^{H} = [5, 10, 20, 50, 200, 400]$, $\boldsymbol{\varepsilon}_{GA}^{H} = [10, 20, 30, 50, 100, 200] $. The corresponding convergence rates are $2.38$ and $2.33$.
\begin{figure}[!ht]
    \includegraphics[width = 0.5\textwidth]{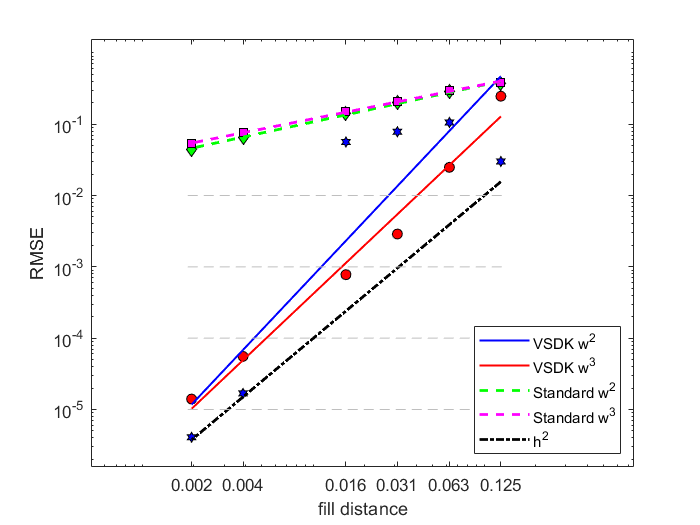}
    \hfill
    \includegraphics[width = 0.5\textwidth]{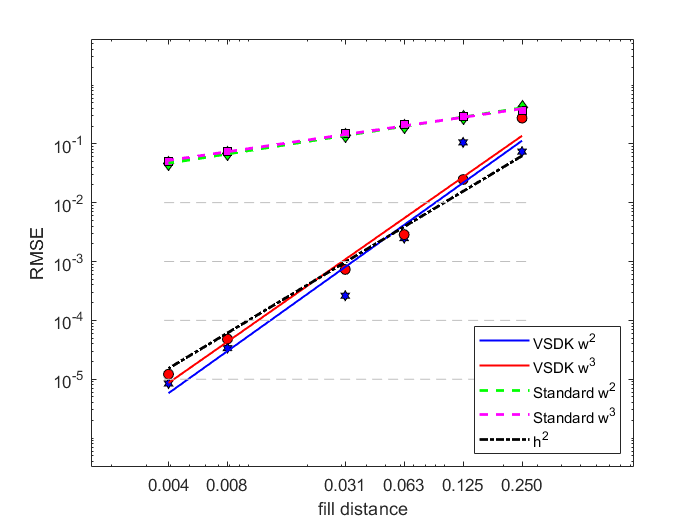}
    \caption{Convergence rates for approximating $f_1$ with MLS-VSDK and MLS-Standard schemes using \textit{uniform} data sites (left) and \textit{Halton} data sites (right).}
   \label{f1;con;rate}
\end{figure}
On the other hand, using non-scaled weight functions, the standard MLS scheme can hardly reach an approximation order of 1, in both cases.

\subsection{Example 2} \label{sub;sec;2}
Consider on $\Omega = (-1,1)^2$ the discontinuous function
\begin{equation*}
    f_{2}(x,y) = \begin{cases}
    \exp(-(x^2+y^2)), \quad & x^2 + y^2 \leq 0.6 \\
    x+y, \quad  & x^2 + y^2 > 0.6 
    \end{cases}
\end{equation*}
and the discontinuous scale function $$ \psi(x, y) = \begin{cases} 
    1, \; \; x^2 + y^2 \leq 0.6 \\
    2, \; \; x^2 + y^2 > 0.6
    \end{cases} $$
As evaluation points, we take the grid of equispaced points with mesh size $1.00e-2$. Figure \ref{f2;RMSE;abs} shows both the \textit{RMSE} and \textit{absolute error} for the classical \textit{MLS} and \textit{MLS-VSDK} approximation of $f_{2}$ sampled from $1089=33^2$ uniform data sites taking $w^{1}$ as the weight function. Figure \ref{f2;RMSE;abs} shows that using classical MLS, the approximation error significantly increases near the discontinuities, while using MLS-VSDK the approximant can overcome this issue.
\begin{figure}[!ht]
   \includegraphics[width=0.475\textwidth]{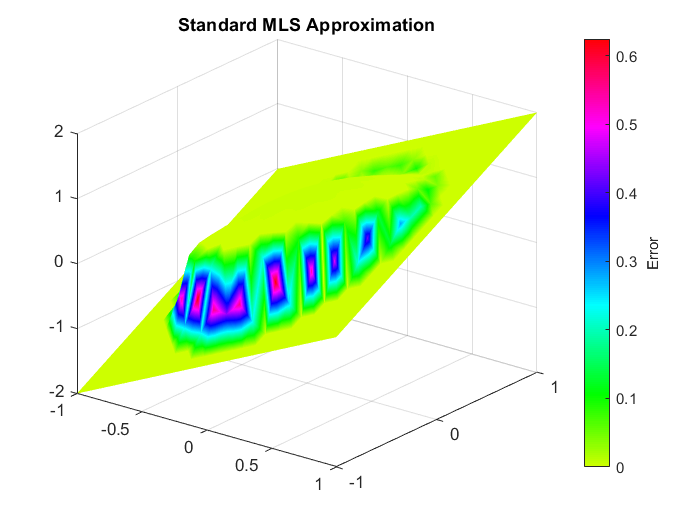}
   \hfill
   \includegraphics[width=0.475\textwidth]{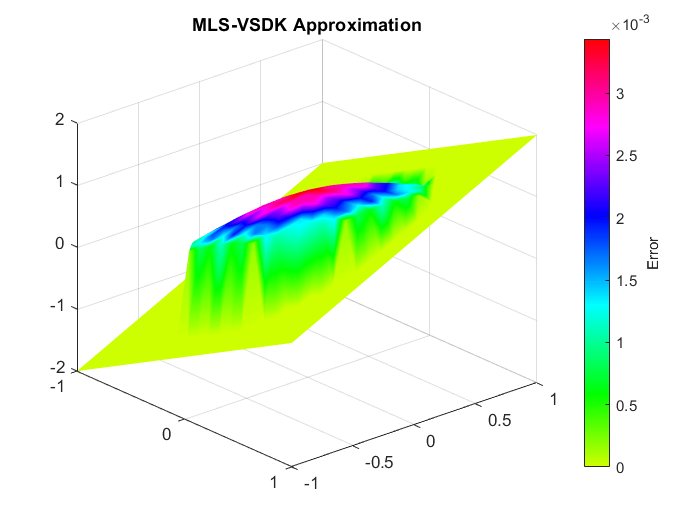}
   \includegraphics[width=0.475\textwidth]{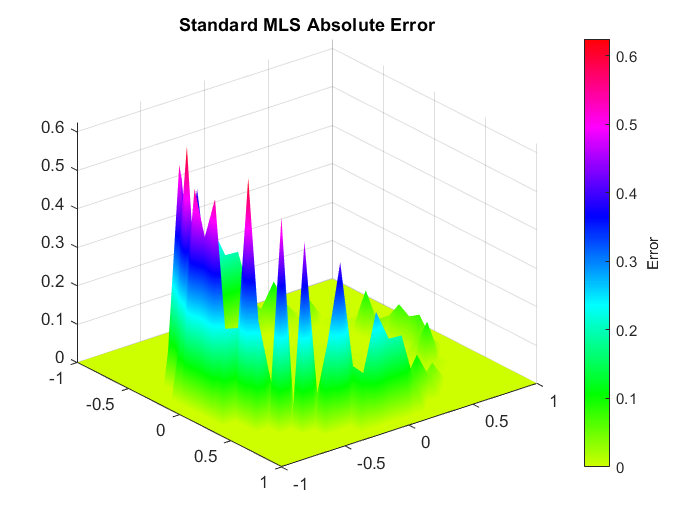}
   \hfill
   \includegraphics[width=0.475\textwidth]{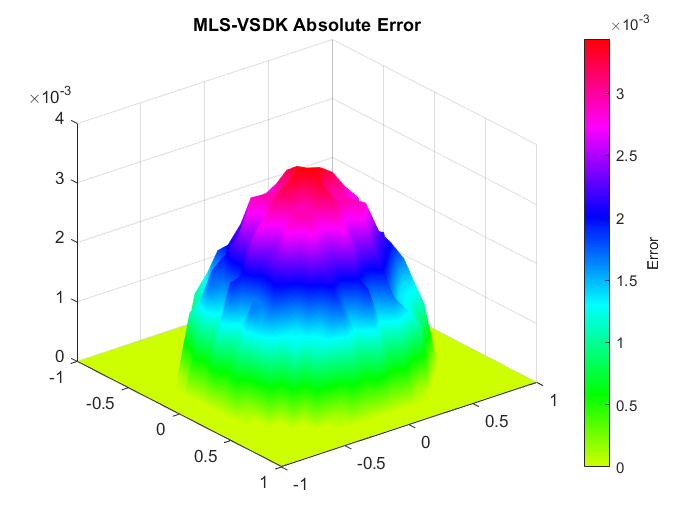}
   \caption{RMSE and abs-error of $f_{2}$ MLS (left) and MLS-VSDK (right) aproximation schemes using $w^1$ weight function}
   \label{f2;RMSE;abs}
\end{figure}
In order to investigate the convergence rate, we consider increasing sets of $\{ 25, 81, 289, 1089,  4225, 16641 \}$ Halton and uniform points as the data sites. To find an appropriate value for the shape parameter, we fix an initial value and  we multiply it by a factor of $2$ at each step. Thus, let $\boldsymbol{\varepsilon} = [0.25, 0.5, 1, 2, 4, 8]$ be the vector of shape parameter which is modified with respect to the number of the centers in both cases of uniform and Halton data sites. The left plot of Figure \ref{f2;con;rate} shows a convergence rate of $2.58$ for the MLS-VSDK and only $0.66$ for classical MLS methods, while these values are $2.04$ and $0.70$ in the right plot.
\begin{figure}[!ht]
   \includegraphics[width = 0.475\textwidth]{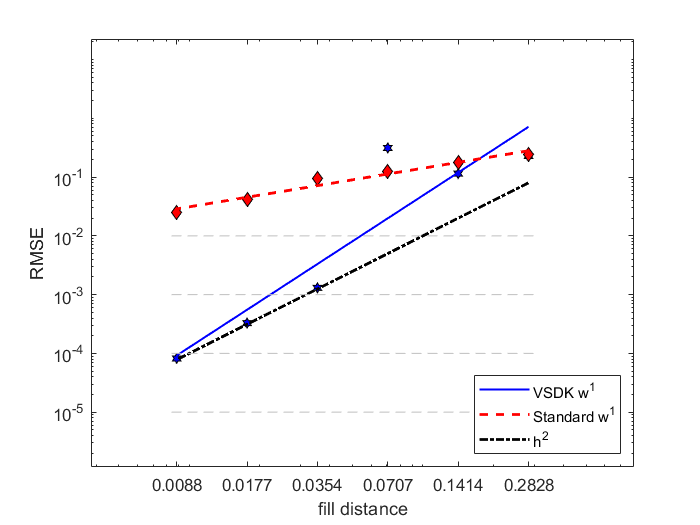}
   \hfill
   \includegraphics[width = 0.475\textwidth]{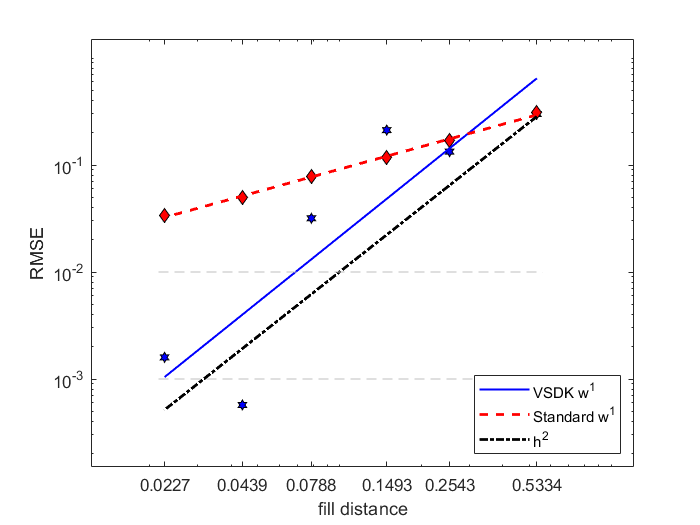}
   \caption{Convergence rates for approximation of function $f_2$ with MLS-VSDK and MLS standard schemes using \textit{Uniform} data sites (left) and \textit{Halton} data sites (right).}
   \label{f2;con;rate}
\end{figure}

\subsection{Example 3} \label{sub;sec;3}
Consider the following function
\begin{equation*}
    f_{3}(x,y) = \begin{cases}
    2 \big( 1 - \exp(-(y+0.5)^2) \big), &\qquad \lvert x \lvert \leq 0.5, \; ,\lvert y \lvert \leq 0.5. \\
    4(x + 0.8), &\qquad -0.8 \leq x \leq -0.65, \lvert y \lvert \leq 0.8. \\
    0.5, &\qquad 0.65 \leq x \leq 0.8, \lvert y \lvert \leq 0.2 \\
    0, &\qquad \text{otherwise}.
    \end{cases}
\end{equation*}
defined on $\Omega = (-1,1)^2$. Regarding the discontinuities of $f_{3}$, the scale function is considered to be
\begin{equation*}
\psi(x, y) = \begin{cases}
    1 , &\qquad \lvert x \lvert \leq 0.5, \; ,\lvert y \lvert \leq 0.5. \\
    2, &\qquad -0.8 \leq x \leq -0.65, \lvert y \lvert \leq 0.8. \\
    3, &\qquad 0.65 \leq x \leq 0.8, \lvert y \lvert \leq 0.2 \\
    0, &\qquad \text{otherwise}.
    \end{cases}
\end{equation*}
Moreover, let the centers and evaluation points be the same as the Example \ref{sub;sec;1}. Table \ref{f3;table} and \ref{f3;table2} shows RMSE of MLS-VSDK and conventional MLS approximation of $f_{3}$ using $w^4$ which interpolates the data. We underline that our experiments show that the stencil of size $n=20$ leads to the best accuracy.
\begin{table}[!ht] 
\begin{tabular}{ || c | c |  c |  c || }
\hline
\textbf{number of centers}	& \textbf{$\varepsilon$ value} & \textbf{RMSE MLS-VSDK}  &\textbf{RMSE classic MLS}\\
\hline \hline
25		& 1			& 3.67e-1      & 1.47e+0  \\
81		& 2			& 3.68e-1      & 8.86e-1  \\
289     & 4         & 1.49e-2      & 7.44e-1  \\
1089    & 8         & 4.23e-3      & 7.72e-1  \\
4225    & 16        & 1.06e-3      & 6.64e-1  \\
16641   & 32        & 2.65e-4      & 5.25e-1 \\
\hline
\end{tabular}
\centering
\caption{RMSE of $f_{3}$ interpolation with \textit{uniform} data sites.\label{f3;table}}
\centering
\end{table}
\begin{table}[!ht] 
\begin{tabular}{ || c | c |  c |  c || }
\hline
\textbf{number of centers}	& \textbf{$\varepsilon$ value} & \textbf{RMSE MLS-VSDK}  &\textbf{RMSE classic MLS}\\
\hline \hline
25		& 1			& 8.84e-1      & 1.53e+0  \\
81		& 2			& 8.95e-2      & 1.05e+0  \\
289     & 4         & 1.42e-2      & 8.74e-1  \\
1089    & 8         & 4.18e-3      & 6.48e-1  \\
4225    & 16        & 1.09e-3      & 6.68e-1  \\
16641   & 32        & 3.02e-4      & 7.07e-1 \\
\hline
\end{tabular}
\centering
\caption{RMSE of $f_{3}$ interpolation with \textit{Halton} data sites.\label{f3;table2}}
\centering
\end{table}
Figure \ref{f3;RMSE;abs} shows \textbf{RMSE} and \textbf{Absolute Error} for \textit{standard MLS} and \textit{MLS-VSDK approximation} of $f_{3}$ sampled from $1089$ uniform points using $w_{4}$ as weight function. Once again, Figure \ref{f3;RMSE;abs} shows how MLS-VSDK scheme can improve the accuracy by reducing the error near the jumps.
\begin{figure}[!hbt]
   \includegraphics[width=0.475\textwidth]{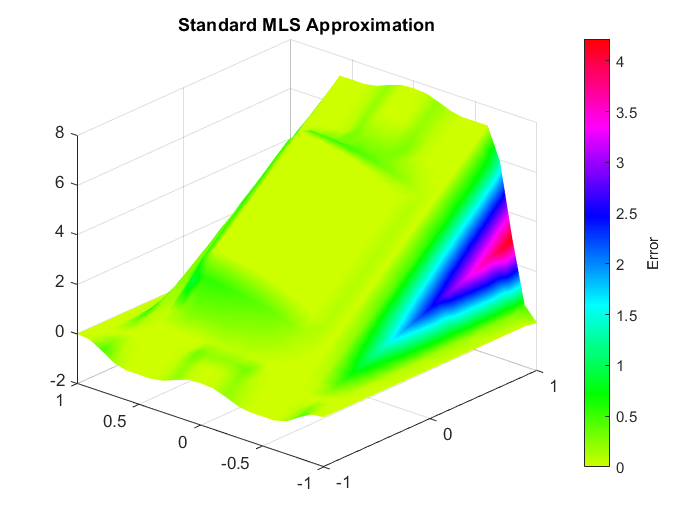}
   \hfill
   \includegraphics[width=0.475\textwidth]{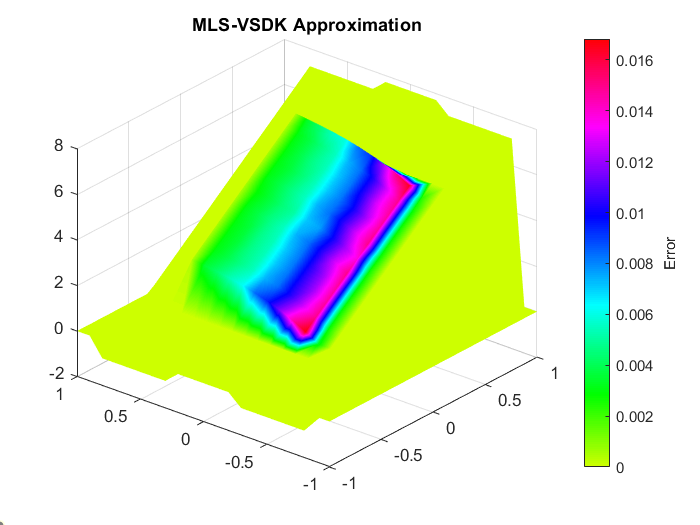}   \includegraphics[width=0.475\textwidth]{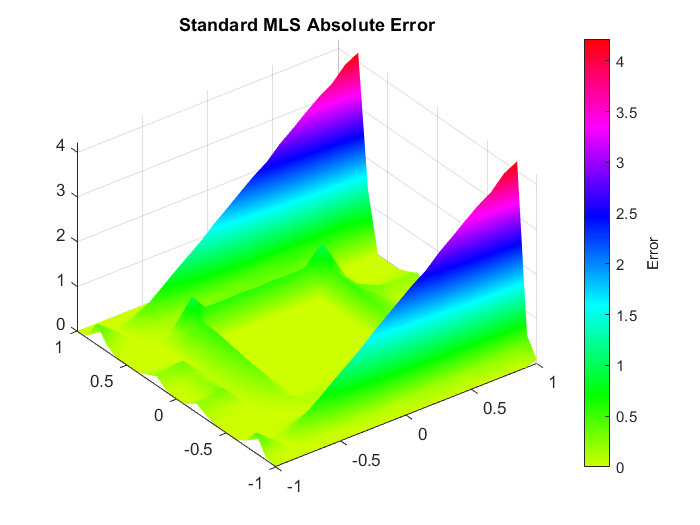}
   \hfill
   \includegraphics[width=0.475\textwidth]{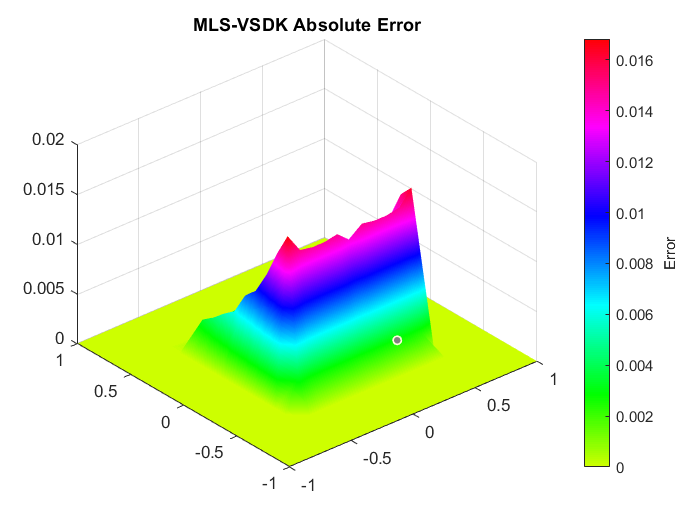}
   \caption{RMSE and abs-error of $f_{3}$ MLS(left) and MLS-VSDK(right) aproximation(interpolation) schemes using $w^4$ weight function}
   \label{f3;RMSE;abs}
\end{figure}
Eventually, letting $\boldsymbol{\varepsilon}_{GA}^{U} = [2, 4, 8, 16, 32, 64]$ and $\boldsymbol{\varepsilon}_{Mat}^{U} = [10, 20, 40, 80, 160, 320]$, Figure \ref{f3;con;rate} shows that $h^2$ convergence is achievable. To be more precise, the rate of convergence in the left plot is $2.54$ and $2.69$ for $w_{2}$ and $w_{3}$, respectively. On the other hand, letting $\boldsymbol{\varepsilon}_{GA}^{H} = [1, 2, 4, 8, 16, 32]$ and $\boldsymbol{\varepsilon}_{Mat}^{H}$ as the Uniform case, convergence rates of $2.50$ and $2.73$ is achievable when \textit{Halton} data sites are employed.
\begin{figure}[!ht]
   \includegraphics[width = 0.47\textwidth]{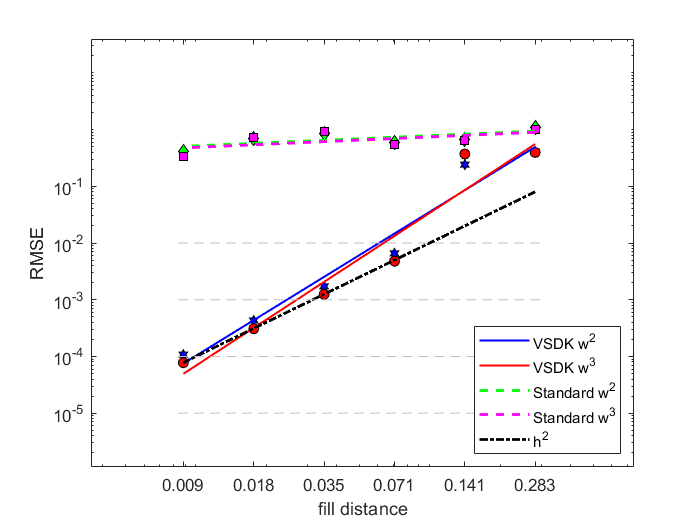}
   \hfill
   \includegraphics[width = 0.47\textwidth]{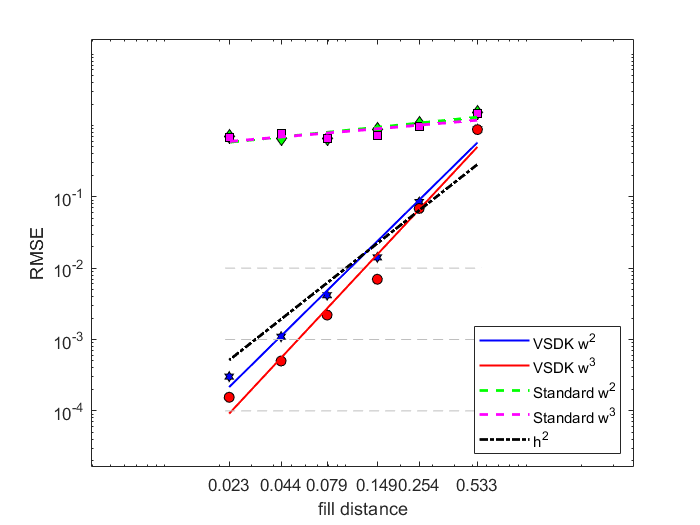}
   \caption{Convergence rates for approximation of function $f_3$ with MLS-VSDK and MLS standard schemes using \textit{Uniform} data sites (left) and \textit{Halton} data sites (right).}
   \label{f3;con;rate}
\end{figure}
\subsection{Example 4} \label{sub;sec;4}
In applications, the discontinuities are likely to be unknown. To overcome this problem, one can consider edge detector method to extract the discontinuities. However, in this way the approximation depends also on the performance of the edge detector method  as well \cite{ref-journal12}. In this direction, in this final experiment the location of the discontinuities are not exact. This is modeled by adding some noise drawn from the standard normal distribution multiplied by $0.01$ to the edges of $\Omega_{\ii} \in \mathcal{P}$. We take the test function $f_{2}$ and the data sites in Section \ref{sub;sec;2}. We fix $n=25$, and $\boldsymbol{\varepsilon}_{GA} = [0.25, 0.5, 1, 2, 4, 8]$, $\boldsymbol{\varepsilon}_{Mat} = [1, 2, 4, 8 16, 32]$ for both {\it Halton} and {\it uniform} centers. Figure \ref{f2;con;rate;noise} shows that the suggested MLS-VSDK is still able to obtain a good convergence rate when compared to classical MLS even when the discontinuities are nor known exactly.
\begin{figure}[!ht]
   \includegraphics[width = 0.47\textwidth]{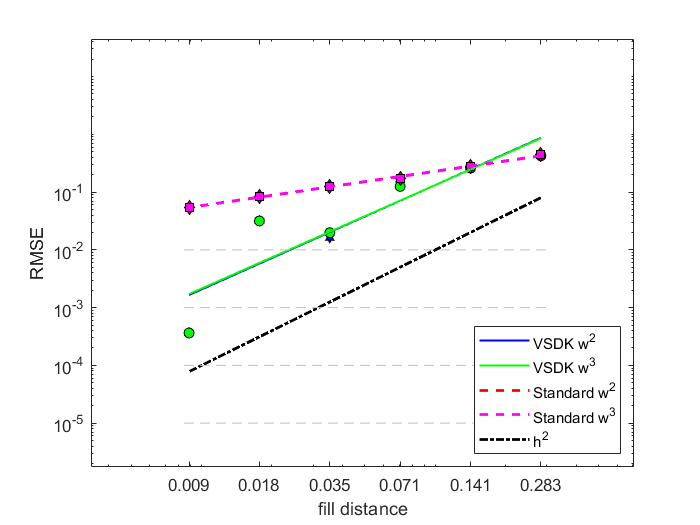}
   \hfill
   \includegraphics[width = 0.47\textwidth]{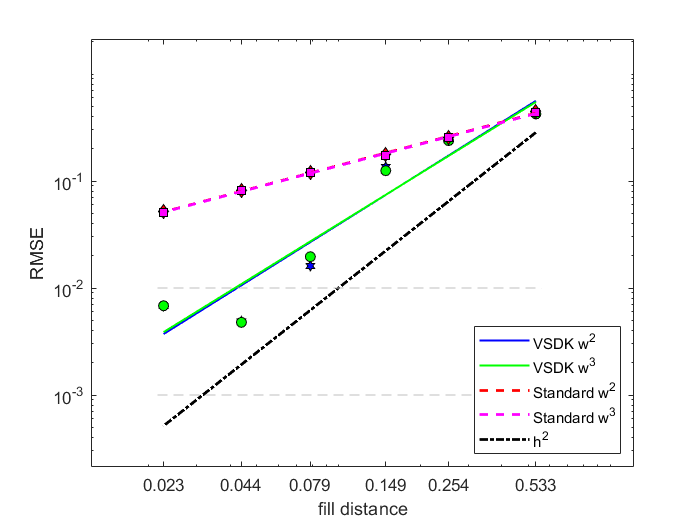}
   \caption{Convergence rates for approximation of function $f_2$, based on noisy given data values, with MLS-VSDK and MLS standard schemes using \textit{Uniform} data sites (left) and \textit{Halton} data sites (right).}
   \label{f2;con;rate;noise}
\end{figure}

\section{Conclusions}\label{sec5}
To approximate a discontinuous function using scattered data values, we studied a new technique based on the use of discontinuously scaled weight functions, that we called the MLS-VSDK scheme, that is the application of discontinuous scaled weight functions to the MLS. It enabled us to move toward a data-dependent scheme, meaning that MLS-VSDK is able to encode the behavior of the underlying function. We obtained a theoretical Sobolev-type error estimate which justifies why MLS-VSDK can outperform conventional MLS. The numerical experiments confirmed the theoretical convergence rates. Besides, our numerical tests showed that the suggested scheme can reach high accuracy even if the position of the data values are slightly perturbed.

\vskip 0.1in
\noindent {\bf Acknowledgments.} This research has been accomplished within the Rete ITaliana di Approssimazione (RITA) and the thematic group on Approximation Theory and Applications of the Italian Mathematical Union. We also received the support of GNCS-IN$\delta$AM.



\begin{thebibliography}{999}
\bibitem{Bayona}
Bayona, V.; Flyer, N.; Fornberg, B; Barnett, G.A. On the role of polynomials in RBF-FD approximations: II. Numerical solution of elliptic PDEs. {\em Journal of Computational Physics}, {\bf 2017}, {\em 332}: 257-273.
\bibitem{Cuomo}
Cuomo, S.; Galletti, A.; Giunta, G.; Starace, A. Surface reconstruction from scattered point via RBF interpolation on GPU. {\em 2013 Federated Conference on Computer Science and Information Systems, Krakow, Poland}, {\bf 2013}: 433-440.
\bibitem{Guastavino}
Guastavino, S.; Benvenuto, F. Convergence Rates of Spectral Regularization Methods: A Comparison between Ill-Posed Inverse Problems and Statistical Kernel Learning. {\em SIAM J. Numer. Anal.}, {\bf 2020}, {\em 58(6)}: 3504–3529.
\bibitem{Lucy}
Lucy, L.B. A numerical approach to the testing of the fission hypothesis. {\em The astronomical journal}, {\bf 1977}, {\em 82}: 1013-1024.
\bibitem{Nguyen}
Nguyen, V. P. et al. Meshless methods: a review and computer implementation aspects. {\em Math. Comput. Simul}, {\bf 2008}, {\em 79(3)}: 763-813.
\bibitem{ref-proceeding}
Shepard, D. A two-dimensional interpolation function for irregularly-spaced data. Proceedings of the 1968 23rd ACM national conference, New York, U.S.A, 27-29.Aug/1968.
\bibitem{ref-journal7}
Lancaster, P.; Salkuaskas, K. Surfaces generated by moving least squares methods. {\em Math. Comput.} {\bf 1981}, {\em 37}, 141--158.
\bibitem{ref-journal5}
Bos, L.; Salkauskas, K. Moving least-squares are Backus-Gilbert optimal. {\em J. Approx. Theory.} {\bf 1989}, {\em 59}, 267--275.
\bibitem{ref-journal15}
Backus, G.E.; Gilbert, J.F. Numerical Applications of a Formalism for Geophysical Inverse Problems. {\em} {\bf 1967}, {\em 13}, 247--276.
\bibitem{ref-journal13}
Mirzaei, D.; Schaback, R.; Dehghan, M. On generalized moving least squares and diffuse derivatives. {\em SIAM J. Numer. Anal.} {\bf 2012}, {\em 32}, 983 -- 1000.
\bibitem{ref-journal14}
Mirzaei, D.; Schaback. Direct Meshless. Local Petrov–Galerkin (DMLPG) method: A generalized MLS approximation. {\em Appl. Numer. Math.} {\bf 2013}, {\em 68}, 73--82.
\bibitem{ref-journal2}
Levin, D. The approximation power of moving least-squares. {\em Math. Comp} {\bf 1998}, {\em 67}, 1517--1531.
\bibitem{ref-book1}
Wendland, H. \textit{Scattered Data Approximation}, 1st ed.; Cambridge University Press: Cambridge, UK, 2005; p. 336.
\bibitem{ref-journal3}
Wendland, H.  Local polynomial reproduction and moving least squares approximation. {\em SIAM J. Numer. Anal.} {\bf 2001}, {\em 21}, 285 -- 300.
\bibitem{ref-journal8}
Armentano, M.G.; Duran. R.G. Error estimates for moving least square approximations. {\em Appl. Numer. Math.} {\bf 2001}, {\em 37}, 397--416.
\bibitem{ref-journal9}
Armentano, M.G. Error estimates in Sobolev spaces for moving least square approximations. {\em SIAM J. Numer. Anal.} {\bf 2001}, {\em 39}, 38--51.
\bibitem{ref-journal10}
Mirzaei, D. Analysis of Moving Least Square Approximation revisited. {\em J. Comput. Appl. Math.} {\bf 2015}, {\em 282}, 237--250.
\bibitem{ref-journal11}
Narcowich, F.J; Ward, J.D.; Wendland. H. Sobolev Bounds On Functions With Scattered Zeros, With Applications To Radial Basis Function Surface Fitting. {\em Math. Comput.} {\bf 2005}, {\em 78}, 743--763.
\bibitem{ref-journal1}
De Marchi, S.; Marchetti, F.; Perracchione, E. Jumping with Variably Scaled Discontinuous Kernels. {\em BIT Numer. Math.} {\bf 2019}, {\em 60}, 441--463.
\bibitem{ref-journal4}
Bozzini, M.; Lenarduzzi, L.; Rossini, M.; Schaback, R. Interpolation with Variably Scaled Kernels. {\em SIAM J. Numer. Anal.} {\bf 2015}, {\em 35}, 199–-219.
\bibitem{ref-journal12}
De Marchi, S.; Erb. W.;, Marchetti, F.; Perracchione, E.; Rossini, M. Shape-Driven Interpolation with discontinuous Kernels: Error Analysis, Edges Extraction and Application in Magnetic Particle Imaging. {\em J. Sci. Comput.} {\bf 2020}, {\em 42}, 472--491.
\bibitem{ref-book2}
Fasshauer, G.E. \textit{Meshfree Approximation Methods}, 1st ed.; World Scientific Publishing: Singapore, 2007; p. 500.
\bibitem{ref-book3}
Adams R.A.; Fournier, J. \textit{Sobolev Spaces}, 2nd ed.; Elsevier: London, U.K, 2003; p. 305.
\bibitem{ref-thesis1}
C. Rieger, B. Zwicknagl; { \it Sampling inequalities for infinitely smooth functions, with applications to interpolation and machine learning}. Advances in Computational Mathematics 32.1 (2010): 103-129.
\bibitem{ref-journal6}
Bayona, V. Comparison of Moving Least Squares and RBF+poly for Interpolation and Derivative Approximation. {\em J. Sci. Comput.} {\bf 2019}, {\em 81}, 486–-512.
\bibitem{ref-book4}
Bernard, S.C.; Scott, L.R. \textit{The Mathematical Theory of Finite Element Methods}, 3rd ed.; Springer: 2003; p. 397.
\bibitem{ref-book5}
Fasshauer, G.E.; McCourt, M.J.\textit{Kernel Based Approximation Methods Using MATLAB}, 1st ed.; World Scientific Publishing: Singapore, 2015; p. 537.

\end{thebibliography}
\end{document}